\documentclass[11pt,a4paper]{amsart}
\usepackage{graphicx}
\usepackage{amsmath,amssymb,amsfonts,epsfig,amscd}
\usepackage{pdfsync}
\usepackage[latin1]{inputenc}
\usepackage{enumerate}
\usepackage{hyperref}

\hypersetup{ 
    colorlinks=true, %set true if you want colored links
    linktoc=all,     %set to all if you want both sections and subsections linked
    linkcolor=blue,  %choose some color if you want links to stand out
}
\newcommand{\plans}{\mathcal{A}_{3,k} }

\newcommand{\euc}{\mathbb{E}^{3}}
\newcommand{\solida}{\Omega}

\newtheorem{thm}{Theorem}[section]

\newtheorem{lem}[thm]{Lemma}
\newtheorem{prop}[thm]{Proposition}
\theoremstyle{definition}

\theoremstyle{remark}
\newtheorem{rem}[thm]{Remark}

\numberwithin{equation}{section}
\begin{document}

\title[On the solid angle of a convex set] {On the solid angle of  convex sets}

\author[J. Bruna, J.\ Cuf\'{\i}, E.\ Gallego and A.\ Revent\'os]{J. Bruna, J. Cuf\'{\i}, E. Gallego and A. Revent\'os}

\address{%
Departament de Matem\`atiques \\
Universitat Aut\`{o}noma de Barcelona\\\\\
08193 Bellaterra, Barcelona\\
Catalonia}
\email{joaquim.bruna@uab.cat, julia.cufi@uab.cat, eduardo.gallego@uab.cat, agusti.reventos@uab.cat}

\thanks{The authors were partially supported by grants 2021SGR01015 (Generalitat de Catalunya) and  PID2021-125625NB-I00. ( Ministerio de Ciencia e Innovaci\'on).}


\begin{abstract} 
Here we analyze three dimensional analogues  of the classical Crofton's formula for planar compact convex sets. In this formula a fundamental role is played by the visual angle of the convex set from an exterior point. A generalization of the visual angle to convex sets in euclidian space is the visual solid angle. This solid angle, being an spherically convex set in the unit sphere,  has length, area and other geometric quantities to be considered. The main goal of this note is  to express invariant quantities of the original convex set depending on  volume,  surface area and  mean curvature integral by means of integrals of  functions related to the solid angle.
\end{abstract}

\maketitle

\section{Introduction and statement of results} 

The purpose of this note is to give three dimensional analogues of the classical Crofton's formula for a planar compact convex set $K$,
\begin{equation}\label{CCP}
\int_{P\notin K} 2(\omega-\sin\omega)\, dP=L^2-2\pi F.
\end{equation}
Here $L$ is the length of the boundary of $K$, $F$ its area and $\omega=\omega(P)$ is the \emph{visual angle} of $K$ as seen from $P$.

A generalization of the visual angle in the plane  to compact convex sets in the euclidean space  $\euc$ is the {\em visual solid} angle. The solid angle $\Omega(P)$ of a compact convex set $K$ from a point $P\notin K$ is the set of unit directions $u$ such that the ray $P+tu,t\geq 0$ meets $K$. Instead of an arc in the unit cercle $S^1$ and its length $\omega$ we have a (spherically convex) set in the unit sphere $S^2$ with a richer geometry, having area, length, and other geometric quantities that might be considered. 

In our analysis we will be lead to two set functions $\alpha(\Omega), \beta(\Omega)$ defined for  spherically convex sets $\Omega\subset S^2$, both replacing $\omega-\sin\omega$ in different senses. To introduce them we first recall classical notions about spherical geometry, the main reference being \cite{santalo}. The radii perpendicular to the support planes of the cone spanned by $\Omega$ form another cone whose intersection with $S^2$ is the so-called \emph{dual curve} of $\partial \Omega$. The region in $S^2$ bounded by this curve and its symmetrical with respect the origin, which we denote by $\tilde{\Omega}$, consists in the unit directions $v$ such that the plane $v^{\bot}$ meets $\Omega$. We may identify it with the set of great circles in $S^2$ meeting $\Omega$. Its complement $\tilde{\Omega}^c$ in $S^2$ consists of two symmetrical components. Since the scalar product  $\langle v,u\rangle$ does not vanish for $v\in \tilde{\Omega}^c, u\in \Omega$, this quantity has constant sign in each component. The one with positive sign is called the \emph{dual solid cone}, denoted $\Omega^{*}$.

The set functions $\alpha, \beta$ are respectively defined as  

\begin{equation*}\label{alpha}
\alpha(\Omega)=\frac 12\int\limits_{u\in\Omega,\, v\in \tilde{\Omega}}|\langle u, v\rangle|\, du\, dv,
\end{equation*}

\begin{equation}\label{beta}
\beta(\Omega)=\frac 18
\int\limits_{v_i\in \tilde{\Omega}}|\det(v_{1},v_{2},v_{3})|\,dv_{1}\,dv_{2}\,dv_{3},
\end{equation}
where $du$ denotes the Lebesgue measure in $S^2$.

The factors in front of the integrals are explained by the fact that the set $\tilde{\Omega}$ doubly parametrizes planes through the origin meeting $\Omega$, equivalently great circles meeting $\Omega$, so the $v$-integrals are in fact over this set of planes, see below.

It is immediate to check that both set functions are invariant by rigid motions $T$ of the sphere, that is $ \alpha (T(\Omega))=\alpha(\Omega), \beta (T(\Omega))=\beta(\Omega)$.  We point out that a straightforward computation shows  that the formal analogues in the plane
$$
\alpha(I)=\frac 12\int\limits_{u\in I,\, v\in \tilde{I}}|\langle u, v\rangle|\, du\, dv,\quad
\beta(I)=\frac 14
\int\limits_{v_i\in \tilde{I}}|\det(v_{1},v_{2})|\,dv_{1}\,dv_{2},$$
where now $I$ is an interval, both equal $\omega-\sin\omega$ up to a constant, $\omega$ being the length of $I$. 

\medskip
The Crofton type formulas we obtain are then:

\begin{thm}\label{thmalpha}
 For  a compact convex set  $K\subset \euc$ with  mean curvature integral $M$, volume $V$ and surface area $F$ one has
\begin{equation}\label{pairs} 
\frac12 \pi MF-2\pi^2 V=\int_{P\notin K} \alpha(\Omega(P))\,dP.
\end{equation}
\end{thm}

\begin{thm}\label{thmbeta}
For a compact convex set $K$  in $\euc$ with  mean curvature integral $M$ and volume $V$ one has
\begin{equation}\label{triples}
M^{3}-\pi^{4}V=\int_{P\notin K}\beta(\Omega(P))\, dP.
\end{equation}
\end{thm}

\begin{thm}\label{slices}
With the same notations,
\begin{equation}\label{LL}
\int L(K\cap E)^2\, dE=\int_{P\notin K}|\Omega(P)|^2\, dP+4\pi^2 V,
\end{equation}
where $dE$ is the invariant measure for affine planes $E$ in space, $L$ denotes length and $|\Omega|$ surface measure.\footnote{Formula \eqref{LL} was previously considered in private conversations with  E. Teufel in 2003.}
\end{thm}

As a  consequence of Theorem \ref{thmalpha} and Minkowski's inequality $12 \pi V\leq MF$ (see \cite{minkowski}) one has
$$\int_{P\notin K} \alpha(\Omega(P))\,dP\geq 4\pi^2 V,$$
with equality only when  $K$ is a ball.

\medskip
Theorem \ref{slices} has as a consequence:

\begin{thm}\label{bola2} For all convex sets
\begin{equation}\label{bbola2}
\int_{P\notin K} |\Omega(P)|^{2}\, dP\geq 4\pi^2 V,
\end{equation}
and equality holds if and only if  $K$ is a ball.
\end{thm}

Part of our analysis will consist in understanding the set functions $\alpha, \beta$ in terms of metric properties of $\Omega$. For the set function $\alpha$ a satisfactory description is easily obtained, namely
\begin{equation}\label{alphaex}
\alpha(\Omega)=\pi |\Omega|-\langle c(\Omega), c(\Omega^{*})\rangle,
\end{equation}
where $c(\Omega)=\int_{\Omega} u\, du$ and similarly $c(\Omega^{*})$ are (unweighted) centroids. Again, the analogue of this expression for an arc $I$ in $S^1$ ($I^{*}$ being in this case the concentric arc with length $\pi-\omega$) equals $\omega-\sin\omega$ up to constants. It is possible to express everything just in terms of $\Omega$. The explicit computation of $\alpha(\Omega)$ is possible just for spherical caps and other simple cases. 

For the set function $\beta$ the description is not so neat. In fact we show that the difference $\beta(\Omega)-\frac{\pi^2}{2}\alpha(\Omega)$ can be expressed in terms of the {\em dihedral visual angles} ${\mathcal D}(\Omega,u)$ of $\Omega$ from points $u\in S^2$ not in $\Omega$, the analogue in spherical geometry of the visual angle. Note that ${\mathcal D}(\Omega,u)$ is the angle between two planes through the origin and $u$ tangent to $\Omega$. From this it can be seen that a linear combination of equations \eqref{pairs} and \eqref{triples} is the classical  Crofton-Herglotz formula, so Theorem \ref{thmalpha}  and Theorem \ref{thmbeta} can be seen as equivalent statements (see section \ref{2402}).

Regarding theorem \ref{slices}, it would be interesting to understand the left-hand side of \eqref{LL} in terms of the geometry of $K$. 

\section{Proofs of theorems}

\subsection{ Proof of theorems \ref{thmalpha}, \ref{thmbeta}, \ref{slices} and \ref{bola2}} All of them are obtained by  mimicking the integral geometry proof of the plane Crofton formula \eqref{CCP}. We denote by $E,G$ affine planes and lines in space, respectively, and by $dE, dG$ the corresponding canonical invariant measures as used for instance in \cite{santalo}. 

From the classical Crofton's formulas for  intrinsic volumes, mean curvature integral $M$,  area of the boundary $F$ and volume $V$ (see \cite{santalo}) we have,
for  a given compact convex set $K\subset \euc$,  
\begin{equation}\label{croftonG}
\int_{G\cap K\neq\emptyset}dG=\frac{\pi}2 F,\quad \int L(K\cap G)dG=2\pi V, \\
\end{equation}
and
\begin{equation}\label{croftonE}
\int_{E\cap K\neq\emptyset}dE=M, \int L(K\cap E)dE=\frac{\pi^{2}}2F, \int A(K\cap E)dE=2\pi V. 
\end{equation}

We shall consider  pairs and triples of linear varieties provided with the product measure. It will be useful to express these measures using the parametrizations in  next lemma, whose proof can be deduced from section 12 of  \cite{santalo}:

\newpage

\begin{lem}\label{lema1}  
\begin{enumerate}[a)] The following holds:
%a
\item \begin{equation}\label{dgde}
dG\ dE=|\langle u,v\rangle|dG_P\ dE_P \ dP
\end{equation}
where $P=G\cap E$, $v$ is a unit normal direction of $E$,  $u$ is a unit direction of $G$, $dP$ is the Lebesgue measure on  $\euc$, $dG_P$ is the measure of lines in $\euc$ through $P$ and $dE_P$ the measure of planes in $\euc$ through $P$. 
%b
\item \begin{equation}\label{de1de2de3}
dE^1\ dE^2\ dE^3=|\det(v_{1}, v_{2}, v_{3})|dE^1_P\ dE^2_P\ dE^3_P \ dP.
\end{equation}
where $v_{i}$ are the unit  normal directions to $E_{i}$ and $P=E_{1}\cap E_{2}\cap E_{3}.$ 
%c
\item For pairs of lines $G^1, G^2$  intersecting at a point $P\in \euc$ we have
\begin{equation}\label{parellsGaR3}
dG_P^1\, dG_P^2\, dP=dG^1_E\, dG^2_E\, dE
\end{equation} 
 where $dG_E$ is the measure of lines within $E$.
\end{enumerate}
In all statements, we can identify $dG_P, dE_P$ with $\frac 12 du, \frac 12 dv$ respectively.
\end{lem}
\medskip

We can now proceed to prove the announced theorems.
\medskip

{\em Proof of Theorem \ref{thmalpha}}. We consider pairs $(E,G)$ of planes and lines both meeting $K$.
From \eqref{croftonG}, \eqref{croftonE} and \eqref{dgde}, 
$$
\frac{\pi}2 MF=\int\limits_{ G\cap K\neq\emptyset, E\cap K\neq\emptyset} dG\, dE
=\int_{\euc}\int\limits_{ G_P\cap K\neq\emptyset, E_P\cap K\neq\emptyset} |\langle u,v\rangle| dG_P\ dE_P \ dP.
$$
If $P\in K$, there is no restriction on $G_P, E_P$. Since they are doubly parametrized bu $u,v$ respectively,

$$\int |\langle u,v\rangle| dG_P\ dE_P=\frac 14 \int_{u,v\in S^2} |\langle u,v\rangle| du\, dv.$$
Obviously the $v$ integral does not depend on $u$, so the above equals

$$\pi \int |\langle u,v\rangle|\, dv.$$ 
Choosing $u=(0,0,1)$ and computing in spherical coordinates one obtains the value $2\pi^2$.

If $P\notin K$ then $G_P$ is parametrized by $u\in \Omega(P)$ and $E_P$ is doubly parametrized by $v\in \tilde{\Omega}(P)$ whence 
$$\int\limits_{ G_P\cap K\neq\emptyset, E_P\cap K\neq\emptyset} |\langle u,v\rangle| dG_P\ dE_P=\frac 12 \int_{u\in\Omega(P), v\in \tilde{\Omega}(P)} |\langle u,v\rangle|\, du\, dv=\alpha(\Omega(P)),$$
thus proving \eqref{pairs}. \qed
\medskip

{\em Proof of Theorem \ref{thmbeta}}. Here we use  triples of planes meeting $K$ and proceed analogously to the above proof. Using \eqref{croftonE} and \eqref{de1de2de3} it follows

$$M^3=\int_{\euc}\int\limits_{E_P^i\cap K\neq\emptyset} |\det(v_{1}, v_{2}, v_{3})|dE^1_P\ dE^2_P\ dE^3_P \ dP.$$
Again, if $P$ is within $K$, there is no restriction on the $E_P^i$ and
$$\int |\det(v_{1}, v_{2}, v_{3})|dE^1_P\ dE^2_P\ dE^3_P=\frac 18\int_{v_i\in S^2} |\det(v_{1}, v_{2}, v_{3})|\, dv_1\, dv_2\, dv_3.$$
The integral in $v_1,v_2$ is independent of $v_3$ so the above integral equals
$$\frac{\pi}{2}\int |\det(v_{1}, v_{2}, v_{3})|\, dv_1\, dv_2.$$
Choosing $v_3=(0,0,1)$ and computing in spherical coordinates we get the value $\pi^4$.

If $P\notin K$ then $E_P^i$ are doubly parametrized by $v_i\in \tilde{\Omega}(P)$ so that

$$\int\limits_{E_P^i\cap K\neq\emptyset} |\det(v_{1}, v_{2}, v_{3})|dE^1_P\, dE^2_P\, dE^3_P=\beta(\Omega(P)),$$
and \eqref{triples} is proved. \qed
\medskip

{\em Proof of Theorem \ref{slices}}. We use now pairs  of intersecting lines, both meeting $K$. The measure of this set of lines is
$$\int_{P\in\euc, G_P^i\cap K\neq \emptyset} dG_P^1\, dG_P^2\, dP=\int_{\euc}\bigg(\int_{G_P^i\cap K\neq \emptyset}dG_P^1\, dG_P^2\bigg)\, dP.$$
The contribution of $K$ in the $dP$ integral is $(2\pi)^2 V$, while that of $K^c$ is $\int_{P\notin K} |\Omega(P)|^2 \, dP$. 
On the other hand, by \eqref{parellsGaR3}, this equals
$$\int \bigg(\int_{G_1, G_2\subset E, G_i\cap K\neq \emptyset} dG^1_E\, dG^2_E\bigg) dE.$$
Since $\int_{G\subset E, G\cap K\neq \emptyset}\,dG=L(K\cap E)$ (the Cauchy-Crofton formula in the plane $E$), we are done.  \qed

We point out that this argument is equivalent to integration on $E$ of the planar Crofton formula \eqref{CCP} for $E\cap K$ within $E$.

\medskip
{\em Proof of Theorem \ref{bola2}}. Using the isoperimetric inequality in the plane $E$ the left hand side of \eqref{LL} is bigger than
$$4\pi \int_E A(K\cap E)\, dE,$$
which by the last equality in \eqref{croftonE} equals $8\pi^2V$,
thus proving \eqref{bbola2}. If equality holds, then $L(K\cap E)^2=4\pi A(K\cap E)$ for all $E$, then all $K\cap E$ are discs, and this implies easily that $K$ is a ball. \qed

\subsection{On the set function $\alpha$} To prove the relation \eqref{alphaex} just notice that
$$\int_{v\in \tilde{\Omega}}|\langle u, v\rangle|\,  dv=\int_{v\in S^2} |\langle u, v\rangle|\,  dv-2\int_{v\in \Omega^{*}} |\langle u, v\rangle|\,  dv.$$
The first integral does not depend on $u$ and equals $2\pi$, while in the second one $\langle u, v\rangle$ is positive. Altogether,
$$\alpha(\Omega)=\pi\int_{u\in \Omega} du-\int_{u\in\Omega, v\in \Omega^{*}} \langle u, v\rangle\, du\,dv=\pi |\Omega|-\langle c(\Omega), c(\Omega^{*})\rangle. $$	
When $\Omega$ is a  cap in $S^{2}$ with sperical radius $\omega$ then $$\displaystyle \alpha(\Omega)=2\pi^{2} (1-\cos \omega)-\pi^{2} \cos^2\omega \sin^{2}\omega.$$
%\smallskip	

In order to express this function  in terms of $\Omega$ we find a relation between  the centroids of $\Omega$ and $\Omega^{*}$ using a parametrization of the boundary of $\Omega$. We consider the orientation in $\Omega$ given by the unit outward normal to $S^2$ and 
let  $\gamma(t)$, $0\leq t\leq \ell $ be the arc-length parametrization  of $\partial\Omega$ with the induced orientation, so $T=\gamma'$ is the unit tangent. 
\begin{prop}\label{centroidsprop}\mbox{}

\begin{enumerate}[a) ]
\item $$c(\Omega)=\frac 12\int_{0}^{\ell }\gamma(t) \times \gamma'(t) \ dt, \quad c(\Omega^{*})=\frac 12 \int_{0}^{\ell }k_{g}(t) \gamma(t)\ dt,
$$
where $k_{g}(t)$ is the geodesic curvature of $\gamma(t)$.
\item 
$$ c(\Omega)+c(\Omega^{*})=\frac12\int_{0}^{\ell }\gamma'(t)\times \gamma''(t)\ dt = \frac12\int_{0}^{\ell }k(t)\vec{B}(t)\ dt.
$$
where $k(t)$ is the curvature of $\gamma(t)$ and $\vec{B}(t)$ its binormal.
\end{enumerate}
\end{prop}
\begin{proof}
If $u=(x,y,z)$, the first component of $c(\Omega)$ is $\int_{\Omega} x du$, the flow through $\Omega$ of the vector field $X=(1,0,0)$. Since $X=\nabla\times Y, Y= \frac 12 (0,-z,y)$ this component equals
$$\frac 12 \int_{\partial \Omega}\langle T,Y(\gamma(t))\rangle \, dt.$$
Now $\langle T,Y\rangle$ equals the first component of $\gamma\times \gamma'(t)$. Similarly the other components, so the first formula in a) is proved. Next, notice that  $\gamma^{*}(t)=\gamma(t)\times \gamma'(t)$ parametrizes the dual curve, the boundary of $\Omega^{*}$, whence one has as well
$$c(\Omega^{*})=\frac 12 \int_{0}^{\ell }\gamma^{*}(t) \times (\gamma^{*})'(t) \ dt.$$
Now,
$$
\gamma^{*}  \times {\gamma^{*}}'=(\gamma\times \gamma')\times (\gamma\times \gamma')'=(\gamma\times \gamma')\times (\gamma\times \gamma'')=\det(\gamma,\gamma',\gamma'')\gamma=k_{g}\gamma
$$
and a) is proved. 

In order to prove b) we simplify  the notation  writing  $\sigma = \gamma\times \gamma'+\gamma^{*}\times{\gamma^{*}}'.$ Denote by $\vec{T},\vec{N},\vec{B}$  the Frenet frame of $\gamma$. From the definitions it is easy to see that $\langle \sigma, \vec{T}\rangle = 0$. Also
$$
\langle \sigma, \vec{N}\rangle =\frac1k\langle \gamma\times \gamma',\gamma''\rangle +\frac1k\langle \det(\gamma,\gamma',\gamma'')\gamma, \gamma''\rangle=
\frac1k\langle \gamma\times \gamma',\gamma''\rangle -\frac1k \det(\gamma,\gamma',\gamma'')=0
$$
because $\langle \gamma,\gamma'\rangle=0$ and so $ \langle\gamma, \gamma''\rangle=-1.$ Now we compute $\langle \sigma,\vec{B}\rangle,$
\begin{eqnarray*}
\langle \sigma,\vec{B}\rangle & =& \langle \sigma,\vec{T}\times \vec{N} \rangle = \frac1k\langle \sigma, \gamma'\times \gamma''\rangle  =\frac1k
\langle \gamma\times \gamma'+\gamma^{*}\times{\gamma^{*}}', \gamma'\times \gamma''\rangle=\\
&=& \frac1k\langle \gamma\times \gamma', \gamma'\times \gamma''\rangle +\frac1k\langle \gamma^{*}\times{\gamma^{*}}', \gamma'\times \gamma''\rangle=\\
&=& -\frac1k\langle \gamma, \gamma''\rangle \langle\gamma', \gamma'\rangle+\frac1k\langle k_{g}\gamma, \gamma'\times \gamma''\rangle=
\frac1k(1+k_{g}^{2}).
\end{eqnarray*}
The curve $\gamma$ being  on the unit sphere  we have that $k^{2}=1+k_{g}^{2}$; therefore
$
\langle \sigma,\vec{B}\rangle=k
$
and we conclude that 
$$
\gamma\times \gamma'+\gamma^{*}\times{\gamma^{*}}'=k\vec{B}.
$$
Since $\gamma'\times \gamma''=\vec{T}\times k\vec{N}=k\vec{B}$  the proposition  is  proved.

\end{proof}

%Using these formulas it is easily checked that if $\Omega$ is a spherical cap of angle $\omega$ one has
%$$\alpha(\Omega)=2\pi (1-\cos \omega)-\pi \cos^2\omega \sin\omega.$$

\subsection{On the set function $\beta$}\label{2402}

Here we wish to obtain an alternative expression for \eqref{beta} which will lead us to the Crofton-Herglotz formula. 

 Let $u=v_2\times v_3/|v_2\times v_3|$. To specify a basis for $u^{\bot}$ we write $u$ in spherical coordinates, $u=(\sin\varphi\cos\theta, \sin\varphi\sin\theta, \cos\varphi)$, and define
$$
e_{1}=\frac{\partial}{\partial\varphi}= (\cos\varphi\cos\theta, \cos\varphi\sin\theta, -\sin\varphi), 
\qquad e_{2}=\frac{1}{\sin\varphi}\frac{\partial}{\partial \theta}=(-\sin\theta,\cos\theta,0),$$
so that  $\{e_{1}, e_{2}, v\}$ is a positive  orthonormal basis. We write $v_2, v_3$ in this basis
$$
v_{2}=\cos\theta_{2}\cdot e_{1}+\sin\theta_{2}\cdot e_{2},\quad 
v_{3}=\cos\theta_{3}\cdot e_{1}+\sin\theta_{3}\cdot e_{2}.
$$
Then $v_{2},v_{3}$ are parametrized  by $u, \theta_{2},\theta_{3}$. Keeping in mind that the integral is in fact over the set of triplets of great circles meeting $\Omega$, so $\pm v_2, \pm v_3$ count as one, we let $0\leq \theta_2,\theta_3\leq \pi$ and  require that the angles $\theta_2, \theta_3$  be within the dihedral angle ${\mathcal D}(\Omega,u)$ determined by $\Omega$ and $u$. Then $u\in S^2, \theta_{2},\theta_{3}$ is a double parametrization of the set of pairs of great circles meeting $\Omega$. In this parametrization, it is immediate to check that  (cf.~\cite[(34.1)]{ReyPastor1951})
$$ 
dv_{2}dv_{3}=|\sin(\theta_{3}-\theta_{2})| d\theta_{2}d\theta_{3}dv,
$$
while

$$ 
|\det (v_{1},v_{2},v_{3})|=|v_{2}\times v_{3}|\cdot |\langle v_{1}, u\rangle|=|\sin(\theta_{3}-\theta_{2})|\cdot |\langle v_{1},u\rangle|.
$$
Thus
$$
\beta{\Omega)=\frac 14\int_{ v_{1}\in \tilde{\Omega}, u\in S^2, \theta_{i}\in {\mathcal D}(\Omega,u)}}\sin^{2}(\theta_{3}-\theta_{2})|\langle v_{1},u\rangle| dv_{1}\,d\theta_{2}\,d\theta_{3}\,du
$$

Now the integral with respect to  $\theta_{2}, \theta_{3}$ is easily computed. Denoting as well by  ${\mathcal D}(\Omega,u)$ the measure of the dihedral angle we get
$$\beta(\Omega)=\frac{1}{8}\int_{ v_{1}\in \tilde{\Omega},\, u\in S^2 }({\mathcal D}^{2}(\Omega,u)-\sin^{2}{\mathcal D}(\Omega,u))|\langle v_{1},u\rangle| dv_{1}\,du.$$

For $\pm u\in \Omega$ one has ${\mathcal D}(\Omega,u)=\pi$ whence the contribution of this part equals $\pi^2\alpha(\Omega)/2$. Thus
$$\beta(\Omega)=\frac{\pi^2}{2}\alpha(\Omega)+\gamma(\Omega),$$
with
$$\gamma(\Omega)=\frac 18\int_{ v\in \tilde{\Omega},\, \pm u\notin \Omega }({\mathcal D}^{2}(\Omega,u)-\sin^{2}{\mathcal D}(\Omega,u))|\langle v,u\rangle| dv\,du.$$

Thus theorems \ref{thmalpha} and \ref{thmbeta} imply
$$M^3-\frac 14\pi^3 MF=\int_{P\notin K}\gamma(\Omega(P))\, dP.$$

We now insert the definition of $\gamma(\Omega(P))$ and use \eqref{dgde}.  If $u$ the unit direction of $G, {\mathcal D}(\Omega(P),u)$ is the dihedral angle ${\mathcal D}(K,G)$ of $K$ as seen from $G$ so the right-hand side above equals

$$\frac 12\int_{E\cap K\neq \emptyset, G\cap K=\emptyset} ({\mathcal D}^{2}(K,G)-\sin^{2}{\mathcal D}(K,G)) dE\, dG=$$
$$=\frac 12 \int_{E\cap K\neq \emptyset} dE\bigg(\int_{G\cap K=\emptyset} ({\mathcal D}^{2}(K,G)-\sin^{2}{\mathcal D}(K,G)\bigg) \, dG.
$$
Using \eqref{croftonE} we obtain the classical Crofton-Herglotz formula  [\cite{Blaschke2007} p.75] 
$$\int_{G\cap K=\emptyset} ({\mathcal D}^{2}(K,G)-\sin^{2}{\mathcal D}(K,G)) \, dG=2M^2-\frac{\pi^3 F}{2}.$$

This shows that in the presence of Crofton-Herglotz formula, theorems \ref{thmalpha} and \ref{thmbeta} can be seen as equivalent statements.

\section{Some inequalities for convex sets of constant width}
In this section we will deal with compact convex sets $K$ of constant width.  For each of these sets we have  the relation
$R+r=a,$
where $a$ is the width of $K$, and  $r,R$ are respectively the inradius and the circumradius of $K$. Thus, denoting $c=r/R$, we have 
$$r=\frac{ac}{1+c}, \quad R=\frac{a}{1+c}.$$
From Jung's theorem (\cite[3.4.2]{Martini2019}) it follows that $c\geq \sqrt{8/3}-1=0.63...$ 
\begin{thm}\label{ineq}
Let $K$ be a compact convex set of constant width $a$ and  $c=r/R$ the quotient between the inradius and the cirumradius of $K$. Then 
\begin{eqnarray}\label{1612a}\int L(K\cap E)^{2}dE\leq 8\pi^{3}a^{3}\left(\frac{1}{(1+c)^{2}}-\frac{1}{12}\right)
\,,\end{eqnarray} which is an equality for spheres.
\end{thm}

\begin{proof}
First we observe that denoting by $p(u)$, $u\in S^{2}$, the support function of $K$ and $\eta(u)=p(u)-a/2$ one has
$$\eta(u)^{2}=p(u)^{2}+a^{2}/4-ap(u).$$
Hence 
$$\int_{S^{2}} \eta(u)^{2}\,du=\int_{S^{2}} p(u)^{2}\,du+\pi a^{2}-2\pi a^{2}\geq 0,$$
and so
\begin{equation}\label{2107a}
\int_{S^{2}} p(u)^{2}\,du\geq \pi a^{2}.
\end{equation}

We have 
 \begin{eqnarray*}
 \int_{\plans}L(K\cap E)^{2}dE = \frac{1}{2}\int_{S^{2}}\int_{0}^{a}L(K\cap E)^{2}\,dt\,du\leq \frac{1}{2}\int_{S^{2}}\int_{0}^{a}L(S_{R}\cap E)^{2}\,dt\,du\leq   \\ 
 \frac{1}{2}\int_{S^{2}}\int_{0}^{a}\left(2\pi \sqrt{R^{2}-(p(u)-t)^{2}}\right)^{2}dt\,du
 = \int_{S^{2}}2\pi^{2}\left( R^{2}a-\frac{a^{3}}{3}+a^{2}p(u)-ap(u)^{2}\right)\,du
 \end{eqnarray*}
\begin{center}
\includegraphics[width=.75\textwidth]{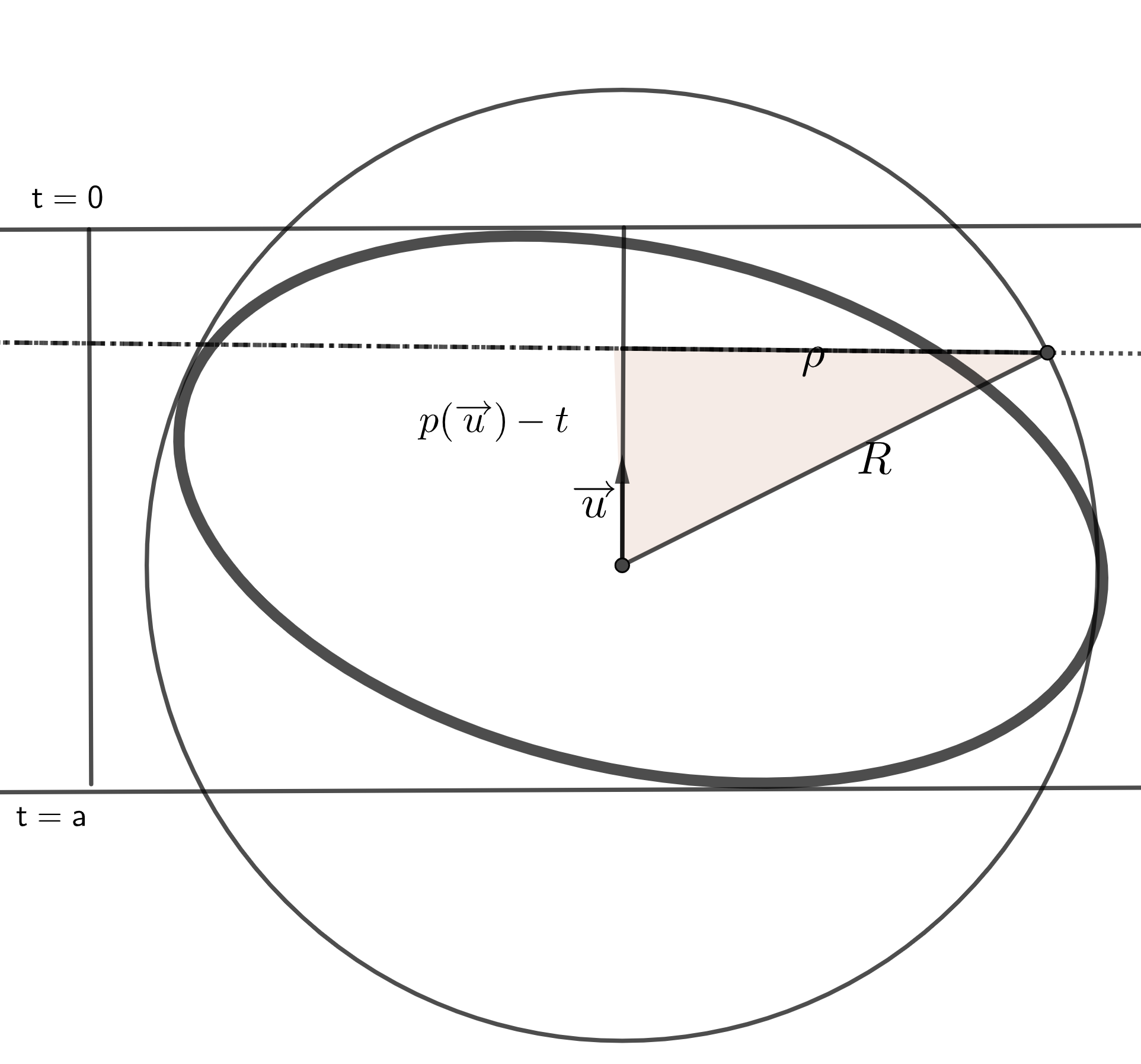}
\end{center}
and by \eqref{2107a} 
\begin{eqnarray*}
\int L(K\cap E)^{2}dE&\leq&8\pi^{3}\left(R^{2}a-\frac{a^{3}}{12}\right)=8\pi^{3}a^{3}\left(\frac{1}{(1+c)^{2}}-\frac{1}{12}\right).
\end{eqnarray*}
\end{proof}

We note that by Jung's inequality $c\geq \sqrt{8/3}-1$,  the above result implies 
\begin{eqnarray*}\label{1212}
\int L(K\cap E)^{2}dE\leq\frac{7}{3}\pi^{3}a^{3}.
\end{eqnarray*}

\begin{prop}Let $K$ be a compact convex set of constant width $a$ with  $c=r/R$ the quotient between the inradius and the cirumradius of $K$. Then 
$$4\pi^{3}a^{3}\bigg(\frac{8}{3}\frac{c^{3}}{(1+c)^{3}}-\frac{1}{6}\bigg)\leq\int_{P\notin K}|\solida (P)|^{2}dP\leq4\pi^{3}a^{3}\bigg(\frac{11-3c(3c^{2}+c-3)}{6(1+c)^{3}}\bigg), $$ with equalities for spheres, where the lower bound is non negative  for $c> 0.657...$.
\end{prop}
\begin{proof}The right hand side inequality comes from \eqref{1612a} and \eqref{LL} substituting $V$ by $V_{r}$,  where $V_{r}$ is the volume of the insphere $S_{r}$ of $K$.  The left hand side  comes subtracting $4\pi^{2}V$ in  the easily checked relations
$$8\pi^{2}V_{r}=\int L(S_{r}\cap E)^{2}dE\leq \int L(K\cap E)^{2}dE$$
and using the inequality $V\leq V_{a/2}$, where $V_{a/2}$ is the volume of the sphere of radius $a/2$ (see \cite{Martini2019}).
\end{proof}

\begin{rem}
We note that in terms of the width only,  we have 
$$\int_{P\notin K}|\solida (P)|^{2}dP\leq \frac{9}{2}\pi^{3}a^{3}(\sqrt{6}-2).$$
\end{rem}
\begin{rem}
 One can ask if  equality 
$$\int L(K\cap E)^{2}dE=\pi MF-4\pi^{2}V$$
that holds for spheres is also true for compact convex sets of constant width. For this case, with the same kind of arguments used above, we are only able to prove 
$$\frac{c^{3}-1}{(1+c)^{3}}\leq \frac{1}{16\pi^{3}a^{3}}\left(\int L(K\cap E)^{2}dE-(\pi MF-4\pi^{2}V)\right)\leq\frac{-23c^{3}+3c^{2}+3c+17}{24(1+c)^{3}} $$
with equalities for spheres (c=1).  
\end{rem}

\bibliographystyle{plain}
%\bibliography{CGRdP}

\end{document}